\documentclass[12pt]{amsart}

\usepackage{amssymb}

\usepackage{enumitem}

\usepackage{graphicx}

\usepackage{mathrsfs}

\newcommand{\h}{\mathbb{H}^\mathrm{Berk}_{\mathbb{R}}}
\newcommand{\x}{\mathrm{x}}
\newcommand{\dg}{\mathtt{deg}}
\newcommand{\A}{\bar{A}}

\newcommand{\V}{\mathcal{V}}
\newcommand{\w}{\varpi}
\newcommand{\K}{\Tilde{K}}
\newcommand{\e}{\eta_{a,r}}

\newcommand\sbmattrix[4]{\textnormal{\scriptsize$\left(\begin{array}{cc}#1&#2\\#3&#4\end{array}\right)$\normalsize}}

\makeatletter
\@namedef{subjclassname@2020}{%
  \textup{2020} Mathematics Subject Classification}
\makeatother

\usepackage[T1]{fontenc}

\newtheorem{theorem}{Theorem}[section]
\newtheorem{corollary}[theorem]{Corollary}
\newtheorem{lemma}[theorem]{Lemma}
\newtheorem{proposition}[theorem]{Proposition}

\theoremstyle{definition}
\newtheorem{definition}[theorem]{Definition}

\newtheorem{example}[theorem]{Example}

\newtheorem*{xrem}{Remark}

\numberwithin{equation}{section}

\begin{document}


\baselineskip=17pt


\title[Continued Fraction in the completion of the Puiseux field]{A Continued Fractions Theory for the completion of the Puiseux field}

\author[Luis Arenas-Carmona]{Luis Arenas-Carmona}
\address{Universidad de Chile, Facultad de Ciencias, Casilla 652, Santiago, Chile.}
\email{learenas@u.uchile.cl}

\author[Claudio Bravo]{Claudio Bravo}
\address{Centre de Mathématiques Laurent Schwartz, École Polytechnique, Institut Polytechnique de Paris, 91128 Palaiseau Cedex.}
\email{claudio.bravo-castillo@polytechnique.edu}

\date{}

\begin{abstract}
In this work, we study a continued fractions theory for the
topological completion of the field of Puiseux series.
As usual, we prove that any element in the completion can be 
developed as a unique continued fractions, whose coefficients are polynomials in roots of the variable, 
and that this approximation is the best 
``rational'' Diophantine 
approximation of such element.
Then, we interpret the preceding result in terms of the action of 
a suitable arithmetic subgroup of the special linear group on 
the Berkovich space defined over the said completion.
We also explore  the connections between points of type IV of 
the Berkovich space in terms of some ``non-convergent'' or
``undefined'' continued fractions, in a sense that we make precise 
in the text.
\end{abstract}

\subjclass[2020]{Primary 11J70, 11J61, 13F25 ; Secondary 20G25, 14G22.}

\keywords{Continued fractions, Diophantine approximation, Puiseux field, modular group, Berkovich space.}

\maketitle

\section{Introduction}

A classical result in number theory, due to Euler, states that any
number $\alpha \in \mathbb{R}$ can be developed in a unique continued
fraction with integral coefficients, i.e. $\alpha$ is the limit of a
unique sequence of the form 
$$[a_0, \dots, a_N]:=a_0+\frac1{a_1+\frac{1}
{\ddots\frac1{a_{N-1}+\frac{1}{a_N}}}},$$
where $N\in \mathbb{Z}_{\geq 0}$, $a_0 \in \mathbb{Z}$ and 
$a_1, \dots, a_N \in \mathbb{Z} \smallsetminus \lbrace 0 \rbrace$. 
This result can be interpreted in terms of the action of 
$\mathrm{SL}_2(\mathbb{Z})$ on the Pioncar\'e half-plain endowed 
with the dual of the Farey tessellation (\cite{Ford,Series}).

In the function field context, an analogous result due to 
Schmidt is proved in \cite{Schmidt}.
This states that, for any field $E$, each $\alpha\in E((t^{-1}))$ 
can be developed in a unique continued fraction with polynomial 
coefficients, i.e. $\alpha$ is the limit of a unique sequence of 
the form $[f_0, \dots, f_N]$ as above, where $f_0 \in E[t]$ and 
$f_N \in E[t]\smallsetminus E$.
The rational approximation so defined is the best Diophantine
approximation of $\alpha$ according to 
\cite[Prop. 2.1 \& 2.2]{Paulin}.
These results have been interpreted in terms of the action of
$\mathrm{SL}_2(E[t])$ on the Bruhat-Tits tree of
$\mathrm{SL}_2\big(E((t^{-1}))\big)$ in \cite{Paulin}.
Some other results on Diophantine approximation for function 
fields have been developed in works like
\cite{Ganguly,Baier,Baier2} by Ganguly, Gosh, 
Bier and Molla, as well as by the authors in 
\cite{ArenasBravoDioph}.

As described above, the continued fraction theory is 
well understood for fields of the form $E((t^{-1}))$, 
and therefore also 
for the extension $E'((t^{-1/n}))$ for any $n$,
as such extensions are all trivially isomorphic.
In particular, the existence of continued fractions approximating elements in the Puiseux field 
$E \langle\langle t^{-1} \rangle \rangle:= 
\bigcup_{n\in \mathbb{Z}_{>0}} E((t^{-1/n}))$ 
is straightforward. 
Such result describes how every
Laurent series on some variable of the form $t^{-1/n}$,
i.e., a series of the form $\sum_{i=1}^\infty a_it^{r_i}$,
where $\{r_i\}_{i=1}^\infty$ is a sequence of rational number
with bounded denominators diverging to $-\infty$, can be written
as a continued fraction of some sort. The field
$E \langle\langle t^{-1} \rangle \rangle$ 
fails to be complete, 
which raises the natural question on what can be said on its 
completion $\widehat{E \langle\langle t^{-1} \rangle \rangle}$.
We answer this question in this work.
As the valuation of these fields is non-discrete, it is to be 
expected that convergence of continued fractions on this setting
is more subtle, and that new phenomena can manifest.

In what follows we refer to a finite sum 
of the form $\sum_{i=1}^N a_it^{r_i}$,
for $a_i\in E$ and $r_i\in\mathbb{Q}_{\geq0}$,
as a Puiseux polynomial.
Its degree is the biggest $r_i \in \mathbb{Q}$ with $a_i \neq 0$.
The ring of all Puiseux polynomials is
denoted $\A=E\langle t\rangle$.
As shown in this work, two phenomena that do appear
in our context are the following:
\begin{itemize}
\item Not every continued fraction whose coefficients are 
positive degree (i.e., non-constant) Puiseux polynomials is convergent.
\item Non-convergent continued fractions are related to some
(but not all) type IV points in a suitable Berkovich space.
\end{itemize}
The usual tools in dealing with continued fractions and their 
dynamical interpretation also need
some sharpening, especially those related to Bruhat-Tits trees,
as the corresponding structures in the present setting are no 
longer graphs in the classical sense. 
We expect that some tools presented here help 
in future works dealing with these and related issues on 
non discretely valued fields.

When $E$ is an algebraically closed field of characteristic $0$, the Puiseux field $E \langle\langle t^{-1} \rangle \rangle$ is the algebraic 
closure of $E((t^{-1}))$ according to \cite[Ch. IV, \S 2, Prop. 8]{Serre-localfields}, so our setting can be see as 
a natural extension of the previously cited results to the 
context of algebraically closed fields.
In \S \ref{sec existence of cf} \& \ref{sec convergence} 
we prove the existence and uniqueness 
of continued fractions whose coefficients are Puiseux 
polynomials. In \S \ref{sec diphantine properties} we study some of their properties as tools for Diophantine approximation.
Then, we interpret the preceding results in terms of the action of a suitable arithmetic subgroup of $\mathrm{SL}_2$, namely the subgroup
of $\A$-rational points,
on the Berkovich space $\mathbb{P}^{1,\mathrm{Berk}}$ defined over $\widehat{E \langle\langle t^{-1} \rangle \rangle}$.
Finally, in \S \ref{sec cf for points of type IV},
 we describe how type IV points
in $\mathbb{P}^{1,\mathrm{Berk}}$ can be described in terms of continued fractions.

\section{Main Results}\label{sec main results}

Consider the field
of rational functions $E(u)$, where $u$ is transcendental over $E$.
Consider also a discrete valuation $\nu: E(u) \to \lambda_{\nu}\mathbb{Z} \cup \lbrace 
\infty \rbrace$, given, for every $f,g \in E[u] \smallsetminus \lbrace 0 \rbrace$,
by $\nu(f/g)= \lambda_\nu\big(\deg(g)-\deg(f)\big)$,
for some fixed constant $\lambda_\nu>0$, and $\nu(0)=\infty$.
Note that the completion $E(u)_\nu$
is the field of Laurent series $K=E((u^{-1}))$.
The discrete valuation on $K$, also denoted by $\nu$, 
is then given by 
$\nu \left(\sum_{i=-N}^{\infty} a_i u^{-i}\right)=N\lambda_\nu$, 
when $a_{-N} \neq 0$.
By setting $u=t^{1/n}$, $K_n=E((t^{-1/n}))$ and $\lambda_\nu=1/n$, 
we can consistently define a valuation on the Puiseux field 
$\K:=E \langle\langle t^{-1} \rangle \rangle=\bigcup_{n=1}^{\infty}K_n$.
We denote by $\hat{K}$ the completion of 
the Puiseux field $\overline{K}$, 
which can be 
characterized as the set of all series $\sum_{i=0}^{\infty} a_i t^{r_{i}}$,
where $r_i$ is a sequence of rational numbers diverging to $-\infty$.
When  $E$ is an algebraically closed field of characteristic $0$, 
the Puiseux field $\K$ is the algebraic closure of $K=K_1$, according to \cite[Ch. IV, \S 2, Prop. 8]{Serre-localfields}.

Let us denote by $A_n=E[t^{1/n}]\subseteq K_n$ the polynomial ring 
in the variable $t^{1/n}$ with coefficients in $E$.
We denote by $\A=E\langle t \rangle$ the union $\bigcup_{n=1}^{\infty} A_n$,
which we call the ring of Puiseux polynomials.
The degree function $\dg$ is defined consistently on the ring $\A$
by the formula $\dg(f)=-\nu(f)$, so that $\dg(t^r)=r$ for 
any $r\in\mathbb{Q}_{\geq0}$. 
In the sequel, we write $k_n=F(t^{1/n})$, and
$\Tilde{k}:=\bigcup_{n=1}^{\infty} k_n=\mathrm{Quot}(\A)$, which is 
strictly contained in the algebraic closure of $k=k_1$.
Note that the ring $\A$ is not principal, or even Noetherian,
but it is still true that every element $x$ in
$\Tilde{k}$, which is contained in some $k_N$, 
can be written as a fraction $x=\frac pq$,
where $p,q\in \A$ span $\A$ as an ideal, and therefore can
have no non-trivial common divisors. 
We call such elements comaximal.

An infinite continued fraction with coefficients 
$\lbrace f_i\rbrace_{i=1}^{\infty} \subset \A$ is an expression of the form:
\begin{equation}\label{eq1}
{\textstyle[f_0, f_1, f_2, f_3, \dots]:=f_0+\frac1{f_1+\frac1{f_2+\frac1{f_3+\dots}}}.}
\end{equation}
We always assume that $f_i\neq0$ for $i>0$. We also consider
finite continued fractions of the following kind:
\begin{equation}\label{eq2}
{\textstyle [f_0, f_1,  \dots,f_{N-1},f_N]:=f_0+\frac1{f_1+\frac{1}{\ddots\frac1{f_{N-1}+\frac{1}{f_N}}}}.}
\end{equation}
Note that the latter has element of $\tilde{k}$
as a value.
We write $\w=[f_0,\dots,f_N]$ for the expression itself,
while we write 
$z=\w_{\mathrm{ev}}=[f_0,\dots,f_N]_{\mathrm{ev}}$
for its value. The length $l(\w)$ of a continued fraction is the index $N$.
Similarly, the length of an infinite continued fraction,
as in (\ref{eq1}), is $\infty$.
We often write an expression like $\w=[f_0,f_1,\dots]$ for a continued fraction
that can be either finite or infinite.
Furthermore,
when the sequence of truncated expressions 
$\w(n)=[f_0, f_1, \dots,f_n]$, for some fixed 
infinite expression 
$\w=[f_0,f_1,\dots]$, satisfies 
$\w(n)_{\mathrm{ev}}\stackrel{n\rightarrow\infty}
\longrightarrow z$,
for some element $z\in\hat{K}$, then we 
write $\w_{\mathrm{ev}}:=z$.
Our first result is the following, which describes any element of $\hat{K} \smallsetminus \Tilde{k}$ as an infinite continued fraction.

\begin{theorem}\label{main teo 1}
For each $z \in \hat{K} \smallsetminus \Tilde{k}$, 
there exists a unique sequence 
$\lbrace f_i \rbrace_{i=0}^{\infty} \subset \A$, 
with $\dg(f_i) >0$ for $i>0$, such that 
$z= [f_0, f_1, \dots,]_{\mathrm{ev}}.$
Moreover, the sequence
$\lbrace f_i \rbrace_{i=0} ^{\infty} 
\subseteq A_M$ precisely 
when $z \in K_M$.
In both cases, if we write 
$[f_0, f_1, \dots, f_n]_{\mathrm{ev}}=p_n/q_n$, 
with $p_n,q_n \in \A$ comaximal, then:
$$ \nu\left(z-\frac{p_n}{q_n}\right)=\dg(f_{n+1})+2\dg(q_n).$$
Furthermore $\dg(q_n)=\sum_{i=1}^n\dg(f_i)$, 
so $\sum_{i=1}^\infty\dg(f_i)$
diverges to $\infty$.
\end{theorem}

We also prove that, for each $z\in \Tilde{k}$ there exists a unique sequence $\lbrace f_i \rbrace_{i=0}^{N} \subset \A$, with $\dg(f_i) >0$ for $i>0$, such that $z= [f_0, f_1, \dots, f_N]_{\mathrm{ev}}$ (see Lemma \ref{mt2b}).
Moreover, the set 
$\lbrace f_i \rbrace_{i=1} ^{N}$ is contained in $A_M$, 
exactly when $z$ belongs to 
$ k_M:=K_M\cap\tilde{k}$.

The fact that the ring of coefficients
for continued fractions on $\hat{K}$
turns out to be $\A$, instead of
the full integral closure of the polynomial ring, 
is to be expected 
since we can write other
algebraic functions as continued
functions on fractional powers
of $t$, as illustrated by
the following expansion
(left), that follows from the
relation on the right:
$$ \sqrt{t+1}=\sqrt t+\frac1{2
\sqrt t+\frac1{2\sqrt t+
\frac1{2\sqrt t+\dots}}},
\qquad
\sqrt{t+1}+\sqrt t=2\sqrt t
+\frac1{\sqrt{t+1}+\sqrt t}.
$$

\begin{theorem}\label{main teo 2}
The continued fraction in (\ref{eq1}) converges whenever the series
 $\sum_{i=1}^\infty\dg(f_i)$ diverges to $\infty$. In particular, every
 continued fraction satisfying such condition corresponds to a different point
 in $\hat{K}$.
\end{theorem}

In \S \ref{sec diphantine properties} we extend a few well known properties 
of continued fractions to the present setting.
For $z=[f_0,\dots]_{\mathrm{ev}}\in\hat{K}$ we write 
$\x_n(z)=[f_0,\dots,f_n]_{\mathrm{ev}}$ 
for the truncated fraction, also known as the 
$n$-th approximant. Then we have next result, 
which generalizes the known fact
that approximants are the best possible rational 
approximations for usual continued fractions:

\begin{theorem}\label{main teo 4}
Let $p,q \in \A$, $q \neq 0$ such that $\nu \left(  
z- \frac{p}{q}\right)>-2\nu(q)$. 
Then, we have $p/q=\x_n(z)$, for some 
$n \in \mathbb{Z}_{\geq 0}$.
\end{theorem}

Theorem \ref{main teo 4} and similar results are a 
straightforward consequence of the corresponding 
properties for polynomial rings. However, we provide
here a direct proof from the techniques used, later on, 
to prove Theorem \ref{main teo 3} below 
(c.f. \S\ref{sec diphantine properties}).

Let $\mathbb{P}^{1,\mathrm{Berk}}$ be the Berkovich projective 
line defined from the valued field $\hat{K}$.
Recall that the points in $\mathbb{P}^{1,\mathrm{Berk}}$
classify into $4$ types according to 
\cite[Ex. 1.4.3]{Berkovich}, as follows:
\begin{itemize}
\item The point of type II and III correspond 
to closed balls
$B[a,r]$, where $a\in \K$ and $r\in \mathbb{Q}$ and 
$r \in \mathbb{R} \smallsetminus \mathbb{Q}$,
respectively.
The set of all points of type II and III is denoted
$\h$ and it has the topology of a generalized tree.  
\item The points of type I correspond to 
the visual limits of rays in
$\h\subseteq\mathbb{P}^{1,\mathrm{Berk}}$, 
i.e., parallelism classes of rays in $\h$.
Such points are naturally in correspondence with
the $\hat{K}$-points of the projective 
line $\mathbb{P}^1$.
\item Finally, the points of type IV which can be 
characterized as  limits of sequences 
$\{x_i\}_{i=1}^\infty \subset
\h$
corresponding to decreasing sequences 
$B_{x_1}\supseteq B_{x_2}\supseteq\cdots$
of closed balls $B_{x_i}=B[a_i,r_i]$
with empty intersection.
Here two such sequences define the same type IV point exactly when they are cofinal.
\end{itemize}

Theorems \ref{main teo 1} and \ref{main teo 2} can 
be interpreted in terms of the action of 
$\Gamma:=\mathrm{SL}_2(\A)$ on 
$\mathbb{P}^{1,\mathrm{Berk}}$.
Indeed, in Prop. \ref{prop fund domain and closing umbrellas}
we prove that the quotient (orbit space) defined from the 
action of $\Gamma$ on 
$\h \subset \mathbb{P}^{1,\mathrm{Berk}}$ 
is isomorphic to a certain ray 
$\mathscr{R}_{\infty} \subset \h$. 
This result is a Berkovich theoretical analog of 
a classical result proved by Serre in 
\cite[Ch. II, \S 2.4, Ex (a)]{SerreTrees},
on the action of a certain modular group on the 
Bruhat-Tits tree.
Thus, given an element $\alpha \in \hat{K}$, we describe 
in \S \ref{sec cf for points of type IV} the promenade 
in $\mathscr{R}_{\infty}$ corresponding to the image in 
$ \Gamma \backslash \h$ 
of the double ray connecting the type I 
points $\alpha \in\hat{K}$ and $\infty$.
Indeed, we prove in Prop. \ref{p 74} that the 
degree of the polynomials $f_n$ in the continued 
fraction converging to $\alpha \in \hat{K}$ can be 
read from the aforementioned promenade,
extending some results of \cite[\S 6]{Paulin} 
to our context.
This construction can be extended so that it makes sense
to talk about continued fractions associated to some, but 
not all, type IV points in $\mathbb{P}^{1,\mathrm{Berk}}$. 
Next result elaborates this notion:

\begin{theorem}\label{main teo 3}
The elements $z\in\hat{K}$ whose corresponding 
continued fraction starts with 
a given sequence $f_0,\dots,f_n$ form a ball in the 
valued field $\hat{K}$.
When $f_0,\dots$ is a sequence for which
$\sum_{i=1}^\infty\dg(f_i)$ converges,
the corresponding sequence of balls has empty 
intersection, and therefore
corresponds to a type IV point in the corresponding 
Berkovich space (see \S\ref{sec geom int of cf}). Not
every type IV point is obtained in this fashion, 
however, those that do not
correspond to finite continued fraction with an 
``undefined'' last coefficient,
in the sense described in Prop. \ref{p613}.
\end{theorem}

\section{Recursive definition of Continued Fractions}
\label{sec existence of cf}

Let $E$ be a field, and write $\overline{E}$
for its algebraic closure. Then $K=E((t^{-1}))$
is a field with a surjective valuation 
$\nu: K \to \mathbb{Z} \cup \lbrace \infty \rbrace$ given by 
$\nu\left(\sum_{i=-N}^{\infty} a_i t^{-i}\right)=N$, 
when $a_{-N} \neq 0$.
In particular $\pi=t^{-1}$ is a uniformizing parameter of $K$.
Next result is a classical theorem attributed to 
Puiseux, but essentially known to Newton.

\begin{lemma}
\cite[Ch. IV, \S 2, Prop. 8]{Serre-localfields}
\label{lemma Puiseux field}
When $E$ has characteristic $0$, the algebraic closure of $K$
is $\overline{E} \langle \langle t^{-1} \rangle \rangle$ 
which equals both
$\bigcup_{n \in \mathbb{Z}_{>0}}\overline{E}((t^{-1/n}))$ 
and $\bigcup_{E'/E \text{ alg}.} 
E' \langle \langle t^{-1} \rangle \rangle$.
\end{lemma}

As before, we write $K_n=E((t^{-1/n}))$
and we extends $\nu$ to all these fields by 
setting $\nu(t^{-1/n})=1/n$.
We denote by $\K=\bigcup_{n>0} K_n =
E\langle \langle t^{-1}\rangle \rangle$ 
the Puiseux field of $E$.
This field is non-complete with respect to a surjective 
valuation $\nu: \K \to \mathbb{Q}\cup \lbrace 0 \rbrace$.
Then, we denote its completion by $\hat{K}$.
By abuse of notation we use $\nu$ for the valuation 
on $\hat{K}$, and therefore also on every subfield.
In all that follows, we keep the notations 
$ A_n=E\left[t^{1/n}\right]$, $k_n=E(t^{1/n})$, $k=k_1$,
 $\A=\bigcup_{n=1}^{\infty} A_n$, 
 $\Tilde{k}=\bigcup_{n=1}^{\infty} k_n$ and 
 $\dg(f)=-\nu(f)$ from the introduction. 
 In particular $\dg(0)=-\infty$.
 We also write $\A_{\neq0}=
 \A \smallsetminus \lbrace 0 \rbrace$.

\begin{lemma}\label{lemma first approx}
Let $z \in \hat{K}^{*}$. Then there exists a 
unique $f \in \A$ such that 
$\nu(z-f) >0$.
Moreover, $f=0$ precisely when $\nu(z)>0$. If $f\neq0$,
then $\dg(f)=-\nu(z)$. In particular, $\dg(f)=0$ 
if and only if $\nu(z) =0$.
\end{lemma}

\begin{proof}
Let $z \in \hat{K}^{*}$. If $\nu(z)>0$ there is nothing 
to prove, whence we assume $\nu(z)\leq0$.
Since $\K$ is dense in $\hat{K}$, there exists 
$z_0 \in \K$ such that $\nu(z-z_0) >0$.
Then $z_0 \in K_n$, for some $n \in \mathbb{Z}_{>0}$.
Hence, we can write 
$z_0= \sum_{i=-N}^{\infty} a_i t^{-i/n}$, where 
$a_i \in E$ and $a_N \neq 0$.
Note that in this case $\nu(z_0)=-N/n \in \mathbb{Q}$.
Hence, if $\nu(z) \leq 0$, then $\nu(z_0) \leq 0$, so 
that $N \geq 0$.
Thus, the polynomial $f=\sum_{i=-N}^{0} a_i t^{-i/n} $ 
belongs to 
$A_n \subset \A$ and satisfies $\nu(z_0-f) >0$.
We conclude $\nu(z-f) \geq \mathrm{min} 
\lbrace \nu(z-z_0),\nu(z_0-f) \rbrace >0$.
Note that $\dg(f)=N=-\nu(z)$.
In particular $\dg(f)=0$ exactly when $\nu(z) =0$.
Moreover, note that, if $f_1, f_2 \in \A$ 
satisfy $\nu(z-f_1) >0$ and 
$\nu(z-f_2) >0$, then $\nu(f_1-f_2) >0$.
In particular, $\dg(f_1-f_2)<0$,
thus $f_1=f_2$, whence the result follows.
\end{proof}

\begin{definition}\label{def Alg of cf}
Let $\mathcal{M}=\lbrace z \in 
\hat{K} : \nu(z)>0\rbrace$, be the maximal 
ideal in the local ring 
$\mathcal{O}=\lbrace z \in \hat{K} : \nu(z)>0\rbrace$,
and write $\mathcal{M}_{\neq0}=\mathcal{M}\smallsetminus\{0\}$.
Let $z \in \hat{K}$. We recursively define a sequence 
$\mathrm{App}(z)=\lbrace (f_0, z_0),(f_1, z_1),
(f_2, z_2),\dots\rbrace$ by the following algorithm:
\begin{itemize}
    \item[\textbf{Step 1:}] Set $z_0=z$, and $n=0$.
    \item[\textbf{Step 2:}] Given $z_n$ find $f_n\in\A$,
    satisfying $\nu(f_n-z_n)>0$, which exists and it is 
    unique according to Lemma \ref{lemma first approx}.
    \item[\textbf{Step 3:}] If 
    $z_n-f_n\in\mathcal{M}_{\neq0}$, 
    set $z_{n+1}=\frac1{z_n-f_n}$. Note that $\nu(z_{n+1})<0$.
    \item[\textbf{Step 4:}] If, on the other hand, 
    $z_n-f_n=0$, do not define
    $z_{n+1}$ and say that the continued fraction ends.
    \item[\textbf{Step 5:}] If the continued fraction 
    has not ended,
    iterate from step (2) for the next value of $n$.
\end{itemize}
\end{definition}
Next result is immediate from the definition:
\begin{lemma}\label{p24}
The sequence $\mathrm{App}(z)$ thus defined can be easily seen to satisfy the following properties.
 \begin{itemize}   
    \item[(i)] $f_0=0$, when $\nu(z)>0$ and $\dg(f_0)=-\nu(z)$, in any other case,
    \item[(ii)]  $\dg(f_i)=-\nu(z_i) >0$, for any $i>0,$
    \item[(iii)]  if the continued fraction ends at any point
    during the process, then $z\in\tilde{k}$, and
    \item[(iv)] if the continued fraction does not end, then we
    obtain an infinite sequence $\lbrace (f_i, z_i) \rbrace_{i=0}^{\infty} \subseteq \A \times \mathcal{M}_{\neq0}$.
    This is the case whenever $z\notin\tilde{k}$.\qed
\end{itemize} 
\end{lemma}

\begin{definition}
If $\w=[g_0,g_1,\dots]$ is any infinite continued fraction, we write 
$\w(n)=[g_0,g_1,\dots,g_n]$ for the truncated expression,
and $\x_n(\w)=\w(n)_{\mathrm{ev}}$ for its value.
By the continued fraction defined by $z$ we mean the expression
$\w_z=[f_0,f_1,f_2,\dots]$, where $\mathrm{App}(z)=\lbrace (f_i, z_i) \rbrace_{i=0}^{\infty}$. The elements $\x_n(z):=\x_n(\w_z)\in\tilde{k}$
are known as the approximants of $z$.
\end{definition}

The remaining part of this section is devoted to prove the following proposition.

\begin{proposition}
\label{prop approx cf}
For each $z\in \hat{K}$ whose continued fraction does not end
before defining $z_{n+1}$, we have the following identity:
\begin{equation}\label{eq approx cf}
\nu\big(z-\x_n(z)\big)=\dg(f_{n+1})+ 2 \sum_{i=1}^{n} \dg(f_i).
\end{equation}
\end{proposition}

In order to prove \ref{prop approx cf} we introduce the following notations:

\begin{definition}
 \label{def rho}
Let $\lbrace f_i \rbrace_{i=1}^n$ as in Def. \ref{def Alg of cf}, and let us denote by $\sigma_i$ the Moebius Transformation given by $\sigma_i(x)= \frac{1}{x-f_i}$, for $x \in \mathbb{P}^1(\hat{K})$.
Let us write $\rho_n= \sigma_n \circ \cdots \circ \sigma_0$. 
\end{definition}

Next result is straightforward:

\begin{lemma}\label{lemma cf in terms of rho}
Let $z\in \hat{K}$ whose continued fraction does not end
before defining $z_{n+1}$, and let 
$\mathrm{App}(z)=\lbrace (f_0, z_0),\dots \rbrace$ be its 
associated sequence. Then, the following identities hold:
\begin{itemize}
\item[(1)] $\rho_n^{-1}(\infty)=\x_n(z)$, and
\item[(2)] $\rho_n(z)=z_{n+1}$.\qed
\end{itemize}
\end{lemma}

\begin{lemma}\label{lemma log derivate and degrees}
In the notation of Lemma \ref{lemma cf in terms of rho}, we have:
\begin{equation}\label{eq relation nu}
\nu \left(\frac{\rho_n(z)}{\rho_n'(z)}\right)= 
\dg(f_{n+1})+ 2 \sum_{i=1}^{n} \dg(f_i).  
\end{equation}
\end{lemma}

\begin{proof}
We proceed by induction on $n$.
Firstly, for $n=0$, we have $\rho_0=\sigma_0$, so that 
$(\rho_0)'(z)=\frac{1}{(z-f_0)^2}$. 
Then $\frac{\rho_0(z)}{\rho_0'(z)}=z-f_0=1/z_1$, whence 
$\nu \left(\frac{\rho_0(z)}{\rho_0'(z)}\right) =\dg(f_1)$, as desired.
Now, assume that Eq. \eqref{eq relation nu} holds for $n \in \mathbb{Z}_{\geq 0}$.
Since $\rho_{n+1}=\sigma_{n+1} \circ \rho_n$, we have 
$\nu \left(\frac{\rho_{n+1}(z)}{\rho_{n+1}'(z)}\right)=
\nu \left(\frac{\sigma_{n+1}(\rho_{n}(z))}{\sigma_{n+1}'(\rho_n(z))
\cdot \rho'_n(z)}\right).$
Since $\sigma_{n+1}'(x)=-\frac{1}{(x-f_{n+1})^2}$, 
the valuation of $\frac{\rho_{n+1}(z)}{\rho_{n+1}'(z)}$ equals 
$ \nu \left( \rho_n(z)-f_{n+1}\right) - \nu\big(\rho_n(z)\big)+\nu\left(\frac{\rho_n(z)}{\rho_n'(z)} \right)$.
Then, it follows from Lemma \ref{lemma cf in terms of rho} and from the inductive hypothesis that:
$$\nu \left(\frac{\rho_{n+1}(z)}{\rho_{n+1}'(z)}\right)= 
\nu \left( z_{n+1}-f_{n+1}\right) - \nu(z_{n+1})+ \dg(f_{n+1}) + 
2 \sum_{i=1}^{n} \dg(f_i).$$
Since $\nu \left( z_{n+1}-f_{n+1}\right)=-\nu(z_{n+2})=\dg(f_{n+2})$ 
and $\nu(z_{n+1})=-\dg(f_{n+1})$, according to 
\textbf{Step 3} in 
Def. \ref{def Alg of cf}
and Lemma \ref{p24}(ii), the result follows.
\end{proof}

\begin{corollary}
In the notation of Lemma \ref{lemma cf in terms of rho},
the following identity holds:
$\nu(\rho_n'(z))=-2\sum_{i=1}^{n+1} \dg(f_i)$.
\end{corollary}

\begin{proof}
We set $\rho_n(z)=z_{n+1}$ in the lemma and use the fact
that $\nu(z_{n+1})=-\dg(f_{n+1})$.
\end{proof}

\begin{lemma}\label{lemma valuation of cf}
Let $\w=[0, f_1, \dots]$ be either a finite or an 
infinite continued fraction
satisfying $\dg(f_i)>0$ for $i>0$. Assume $z=\w_{\mathrm{ev}}\in\hat K$ is
defined, i.e., either $l(\w)$ is finite or 
the sequence $\x_n(\w)$ converges.
Then we have $\nu(z)=-\nu(f_1)=\dg(f_1)> 0$.
\end{lemma}

\begin{proof}
Firstly, assume that $N=l(\w) < \infty$.
We prove the statement by induction on $N \in \mathbb{Z}_{>0}$.
For $n=1$, we have $z=1/f_1 \in k$, and the result is immediate. 
Assume the statement holds for any expression with a
given length $N$.
Let $\w=[0, f_1, \dots, f_{N+1}]$, 
$\w'=[0, f_2, \cdots, f_{N+1}]$ and 
$z=\w_{\mathrm{ev}}\in\hat K$, 
$z'=\w'_{\mathrm{ev}}\in\hat K$.
Then $\nu(z') \geq 0$ by the inductive assumption.
In particular, $\nu(f_1)=\nu(f_1+z')$.
Thus $\nu(z)=\nu\big(1/(f_1+z')\big)=-\nu(f_1)$.
Finally, assume that $N=l(\w)=\infty$.
Since $z$ is the limit of 
$\lbrace \x_n(\w) \rbrace_{n=0}^{\infty}$,
and $\x_n(\w)$ is the evaluation of a finite expression, 
the result follows
from the continuity of the valuation $\nu$ outside of $0$. 
\end{proof}

\begin{proof}[Proof of Prop. \ref{prop approx cf}]
In the notation of Lemma \ref{lemma cf in terms of rho}, 
it suffices to prove that $\nu\big(z-\x_n(z)\big)=
\nu \left(\frac{\rho_n(z)}{\rho_n'(z)}\right)$. 
We proceed by induction on $n \in \mathbb{Z}_{\geq 0}$.
For $n=0$, equality holds by the same argument that
was given in the proof of Lemma 
\ref{lemma log derivate and degrees}.
Now, assume that $\nu\big(z-\x_n(z)\big)=
\nu \left(\frac{\rho_n(z)}{\rho_n'(z)}\right)$, 
for some fixed $n\in \mathbb{Z}_{\geq 0}$, and
for all $z \in \hat{K}$ for which $z_{n+1}$ is defined.
Let us write $\tau=\sigma_0(z)$.
The sequence $\mathrm{App}(\tau)$ is exactly 
$\lbrace (f_1, z_1),(f_2, z_2), \rbrace$, a shift of 
$\mathrm{App}(z)$.
In particular, the Moebius transformation 
$\tilde{\rho}_n=\rho_{n+1} \circ \sigma_0^{-1}$
is precisely the $n$-th term of the sequence given by 
Def. \ref{def rho}, when $z$ is replaced by $\tau$.
Then, it follows from the inductive hypothesis that 
$\nu\big(\tau-\x_n(\tau)\big)=\nu 
\left(\frac{\Tilde{\rho}_n(\tau)}{\Tilde{\rho}_n'(\tau)}\right)$.
On one hand, it follows from Lemma \ref{lemma cf in terms of rho}(1) that $\x_n(\tau)=\sigma_0\big(\x_{n+1}(z)\big)$,
so that $\nu\big(\tau-\x_n(\tau)\big)=\nu\Big(\sigma_0(z)-\sigma_0\big(\x_{n+1}(z)\big)\Big)$, whence
$$\nu\big(\tau-\x_n(\tau)\big)=\nu\big(z-\x_{n+1}(z)\big)-\nu(z-f_0)-\nu\big(\x_{n+1}(z)-f_0\big).$$
Moreover, it follows from Lemma \ref{lemma valuation of cf} that $\nu\big(\x_{n+1}(z)-f_0\big)=-\nu(f_1)=\nu(z-f_0)$.
Hence, we get:
\begin{equation}\label{eq 1 ind}
\nu \left(\frac{\Tilde{\rho}_n(\tau)}{\Tilde{\rho}_n'(\tau)}\right)=\nu\big(\tau-\x_{n+1}(\tau)\big)= \nu\big(z-\x_{n+1}(z)\big)-2\nu(z-f_0).   
\end{equation}
On the other hand, applying the chain rule 
in the denominator, we have 
\begin{equation}\label{eq 2 ind}
\nu \left(\frac{\Tilde{\rho}_n(\tau)}{\Tilde{\rho}_n'(\tau)}\right)=\nu \left(\frac{\rho_{n+1}(z)}{\rho_{n+1}'(z)} \right)+\nu \big( \sigma_0'(z)\big). 
\end{equation}
Moreover, it is straightforward that 
$\nu \big( \sigma_0'(z)\big)=-2\nu(z-f_0)$.
Thus, it follows from Eq. \eqref{eq 1 ind} and 
\eqref{eq 2 ind} that 
$\nu\big(z-\x_{n+1}(z)\big)=\nu 
\left(\frac{\rho_{n+1}(z)}{\rho_{n+1}'(z)}\right)$, 
whence the result follows.
\end{proof}

\section{Proof of Theorems \ref{main teo 1} and \ref{main teo 2}}\label{sec convergence}

Next result is useful in order to prove that the sequence 
$\lbrace \x_n(z) \rbrace_{n=0}^{\infty}$ defined in 
\ref{def Alg of cf} converges to $z$.
It also implies that the expression of $z$ as a 
continued fraction is unique.

\begin{lemma}\label{lemma equal cf}
Let $z,z'$ be two element of $\hat{K}$, and let us write $\mathrm{App}(z)=\{(f_0,z_0),(f_1,z_1),\dots\}$
and $\mathrm{App}(z')=\{(g_0,z'_0),(g_1,z'_1),\dots\}$.
Assume that $z_m$ is defined and
 $\nu(z-z') > 2 \sum_{i=1}^m \dg(f_{i})$, for certain 
 $m \in \mathbb{Z}_{\geq 0}$. Then $z'_m$ is defined
 and $f_i=g_i$, for all $i \leq m$.
\end{lemma}

\begin{proof}
We prove this result by induction on $i \leq m$.
Indeed, since the assumption implies the inequality $\nu(z-z')>0$, 
we have $\nu(z)>0$ exactly when $\nu(z')>0$. 
In other words $f_0=0$ precisely when $g_0=0$.
Furthermore, since $z'-f_0=(z-f_0)+(z'-z)$ and $\nu(z-f_0)>0$, 
we also have $\nu(z'-f_0)>0$. But $g_0$ is the unique element in $
\A$ satisfying $\nu(z'-g_0)>0$ according to Lemma \ref{lemma first approx}.
Thus, we have $f_0=g_0$. 

Setting $n=0$ in Proposition \ref{prop approx cf} we 
have $\dg(f_1)=\nu(z-f_0)$. If we assume that $m\geq 1$, 
then $\nu(z-f_0)$ is strictly smaller than $\nu(z-z')>0$. 
In particular $z'\neq g_0$, so $z'_1$ is defined.
We conclude that 
$\dg(f_1)=\nu(z-f_0)=\nu(z'-g_0)=\dg(g_1)$. 
Now we observe that
$$\sigma_0(z)-\sigma_0(z')=\frac1{z-f_0}-\frac1{z'-g_0}=
\frac{(z-z')-(f_0-g_0)}{(z-f_0)(z'-g_0)},$$
which equals $\frac{z-z'}{(z-f_0)(z'-g_0)}$,
since $f_0=g_0$.
Hence, $$\nu\big(\sigma_0(z)-\sigma_0(z')\big)=
\nu(z-z')-2\deg(f_1)> 2 \cdot \sum_{j=2}^m \dg(f_j)\geq0.$$
In particular, $\nu\big(\sigma_0(z)-f_1\big)>0$ implies
$\nu\big(\sigma_0(z')-f_1\big)>0$, and by uniqueness 
we conclude 
$f_1=g_1$. Note that $\sigma_0(z)=z_1$ and 
$\sigma_0(z')=z'_1$.

Now, assume that $f_j=g_j$, for all $j \leq i < m$. 
By Lemma \ref{lemma cf in terms of rho}, we have $\rho_j(z)=z_{j+1}$
and $\rho_j(z')=z'_{j+1}$, for all $j\leq i-1$.
Then, it is a straightforward
induction to prove that
$\nu(z_{j+1}-z'_{j+1}) >2 \cdot \sum_{t=j+1}^m \dg(f_t)\geq0$,
for said $j$, arguing as in the previous paragraph. In particular
$\nu(z_i-z'_i)>\nu(z_i-f_i)>0$, so $z'_{i+1}$ is defined.
Now we can give one more inductive step and conclude
$\nu(z_{i+1}-z'_{i+1}) >2 \cdot \sum_{t=i+1}^m \dg(f_t)\geq0$.
This also implies that $\nu(z'_{i+1}-f_{i+1})>0$, 
and again  $f_{i+1}=g_{i+1}$ by uniqueness.
\end{proof}

Next we prove the convergence of the continued fraction associated to any
element in $\hat{K}$ whose continued fraction does not end.

\begin{proposition}\label{Prop conv fc}
For each $z \in \hat{K}$ whose continued fraction 
fails to end, the sequence $\sum_{i=1}^{\infty} \dg (f_i)$ diverges to $\infty$.
In particular, the element $z$ equals the limit
$\lim_{n \to \infty} \x_n(z)=[f_0, f_1, \dots ]_{\mathrm{ev}}$
of its approximants.   
\end{proposition}

\begin{proof}
Assume that $\sum_{i=1}^{\infty} \dg (f_i) =\nu<\infty$.
Since $\hat{K}$ is the completion of $\K$, there exists 
$z' \in \K$ such that $\nu(z-z') > 2\nu$.
It follows from Lemma \ref{lemma Puiseux field} that 
$z'\in K_M$, 
for some $M \in \mathbb{Z}_{>0}$.
Thus, it follows from \cite[\S 1]{Schmidt} that there exists 
a sequence 
$\lbrace  g_i \rbrace_{i=0}^\infty \subset A_M$, with 
$\dg(g_i)\geq 1/M$ for $i>0$, such that the sequence 
$\mathrm{y}_m(z'):=[g_0, g_1, \dots, g_m]_{\mathrm{ev}}$
converges to $z'$.
In particular, we have $\sum_{i=1}^{\infty} \dg (g_i) =\infty$.
Now, since for any $n \in \mathbb{Z}_{\geq 0}$ we have
$\nu(z-z')> 2 \nu >  2 \sum_{i=1}^n \dg(f_i)$,
we conclude that $f_i=g_i$, for all $i$.
Thus, we get $\sum_{i=1}^{\infty} \dg (f_i) =\infty$, 
contradicting the assumption.
\end{proof}

\begin{corollary}\label{coro deg with common den}
In the notation of Prop. \ref{Prop conv fc} assume that $\lbrace f_i \rbrace_{i=0}^{\infty}\subseteq A_M$.
Then $z$ belongs to $K_M$.
\end{corollary}

\begin{proof}
Note that, since $f_i \in A_M=F[t^{1/M}]$, any 
$\x_n(z) \in k_M=F(t^{1/M})$.
The result follows since
the completion of $k_M$ with respect to $\nu$ is $K_M$.
\end{proof}
The converse of the preceding result is a direct adaptation of a result of \cite{Schmidt}.

\begin{lemma}\cite[\S 1]{Schmidt}\label{lemma paulin}
For each $z \in K_M\smallsetminus k_M$, 
there exists a unique sequence 
$\lbrace f_i \rbrace_{i=0}^{\infty} \subset A_M$, with 
$\dg(f_i) >0$, for all $i>0$, such that 
$z= [f_0, f_1, \dots, f_n, \dots]_{\mathrm{ev}}.$
Moreover, if $z\in k_M$ then There exists
$f_0,f_1,\dots,f_n$ with $\dg(f_i) >0$, for all $0<i\leq n$,
and $z=[f_0,\dots,f_n]_\mathrm{ev}$.
\end{lemma}

Next result shows that the coefficients of continued 
fractions approximating elements in $\hat{K}$ are unique.

\begin{proposition}\label{prop cf are uniq}
Assume $z=\w_{\mathrm{ev}}=\w'_{\mathrm{ev}}$, for two expression 
of the form $\w=[f_0, \dots]$ and $\w'=[g_0, \dots]$.
Then $l(\w)=l(\w')$ and $f_i=g_i$, for all 
$i\in\mathbb{Z}_{\geq 0}$.
\end{proposition}

\begin{proof}
Without loss of generality we assume that $l(\w)\leq l(\w')$.
Firstly, assume that $l(\w)<\infty$.
By the same argument as used in Lemma \ref{lemma equal cf} we can see that $f_i=g_i$, for all $i \leq l(\w)$.
In particular, if $n=l(\w)<l(\w')$, then 
$0=[0,g_{n+1}, \dots]_{\mathrm{ev}}$. However, 
Lemma \ref{lemma valuation of cf} shows that the valuation of
the right hand side of this identity is $-\nu(g_{n+1})$, which is
a contradiction.
Thus, we conclude that $l(\w')=n$ and $f_i=g_i$, for all $i \leq n< \infty$.
When $l(\w)=\infty$, Lemma \ref{lemma equal cf} directly 
implies that $f_i=g_i$, for all $i\geq 0$, 
which concludes the proof.
\end{proof}

\subparagraph{Proof of Theorem \ref{main teo 1}}
The first statement together with the divergence 
of the series $\sum_{i=1}^\infty\dg(f_i)$ follows
from Prop. \ref{Prop conv fc}. The second statement is Cor. \ref{coro deg with common den} together with Prop. \ref{prop cf are uniq}.
Finally, if $x_n(z)=p_n/q_n$, then Prop. 
\ref{prop approx cf} give us 
$\nu\left(z-\frac{p_n}{q_n}\right)=\dg(f_{n+1})
+2\sum_{i=1}^n\deg(f_i)$. Recall that Lemma 
\ref{lemma cf in terms of rho} tells us that
$\rho_n\big(\x_n(z)\big)=\infty$ and $\rho_n(z)=z_{n+1}$. If we set 
$\rho_n(u)=\frac{r_nu+s_n}{q'_nu-p'_n}$, then
the former identity tells us
that, for some constant $\lambda$, we have
$q_n=\lambda q'_n$ and $p_n=\lambda p'_n$.
From the definition of $\rho_n$, and the properties of 
Moebius transformations, we can assume that
$$\sbmattrix{r_n}{s_n}{q_n'}{-p_n'}=
\sbmattrix011{-f_n}\cdots\sbmattrix011{-f_0}\in
\mathrm{SL}_2(\overline{A}),$$
so in particular $r_np'_n+s_nq'_n=(-1)^n$, and therefore
$r_np_n+s_nq_n=(-1)^n\lambda$. This implies that 
$\lambda\in\overline{A}$, so it is a common divisor
of $p_n$ and $q_n$, and therefore a unit.
Furthermore, $\rho_n'(u)=
\frac{(-1)^n}{(q'_nu-p'_n)^2}=
\frac{(-1)^n\lambda^2}{(q_nu-p_n)^2}$.
Now Lemma
\ref{lemma log derivate and degrees} gives us
$$\nu\left(z-\frac{p_n}{q_n}\right)=
\dg(f_{n+1})+ 2 \sum_{i=1}^{n} \dg(f_i)=
\nu \left(\frac{\rho_n(z)}{\rho_n'(z)}\right)$$
$$=\nu(z_{n+1})+2\nu(q_nz-p_n)=\nu(f_{n+1})+2\nu(q_n)+
2\nu\left(z-\frac{p_n}{q_n}\right).$$
It follows that $\nu\left(z-\frac{p_n}{q_n}\right)=
-\nu(f_{n+1})-2\nu(q_n)=\dg(f_{n+1})+2\dg(q_n)$.
\qed

\medskip

\subparagraph{Proof of Theorem \ref{main teo 2}}
Consider an infinite expression $\w=[f_0,f_1,\dots]$, 
and the corresponding
sequence of finite expressions $\w(n)=[f_0,\dots,f_n]$.
Then applying Prop. \ref{prop approx cf} to the element
$z_m=\w(m)_{\mathrm{ev}}$, for $m>n$, and noting that
$\x_n(z_m)=z_n$, we obtain 
$\nu(z_m-z_n)=\dg(f_{n+1})+2\sum_{i=1}^n\dg(f_i)$.
In particular, if the sum on the right diverges, 
we conclude that
$\{z_n\}_{n=1}^\infty$ is a Cauchy sequence, 
and hence it converges to an 
element $z\in\hat K$. It follows that 
$\w$ is indeed the continued 
fraction associated to $z$. Uniqueness 
follows from Prop. 
\ref{prop cf are uniq}.
The fact that $z\notin\tilde k$ follows 
from the fact that $z\in \tilde{k}$
implies that $z$ is a rational function 
on some element $t^{1/n}$,
and therefore it must have a finite 
expressions as a continued fraction
by Lemma \ref{lemma paulin}.
\qed

\section{On some Diophantine properties}\label{sec diphantine properties}

\begin{definition}
For every element $a\in\hat{K}$ and for every integer $r$,
we denote by $B_a^{(r)}$ the open ball defined by
$B_a^{(r)}=\{b\in\hat{K}|\nu(a-b)>r\}$. Similarly,
we denote by $B_a^{[r]}$ the closed ball defined by
$B_a^{[r]}=\{b\in\hat{K}|\nu(a-b)\geq r\}$.
\end{definition}

\begin{definition}
Consider a expression $\w=[f_0,f_1,\dots,f_n]$. We denote by
$\Omega_\w$ the set of all expressions starting with $\w$,
i.e., expressions of the form 
$\w'=[f_0,\dots,f_n,g_{n+1},\dots]$,
where $g_{n+1},g_{n+2},\dots$ are arbitrary,
and set $B_{\w}=\{\w'_{\mathrm{ev}}|\w'\in\Omega_{\w}\}$.
\end{definition}

\begin{lemma}\label{l42}
    The set $B_\w$ defined above is an open ball. In fact $B_\w=B_a^{(r)}$,
    where $a=\w_{\mathrm{ev}}$ and $r=2\sum_{i=1}^n\dg(f_i)$.
    Furthermore, $B_\w=\eta^{-1}_n\left(B_0^{(0)}\right)$, where
    $\eta_n(z)=1/\rho_n(z)$, according to Definition \ref{def rho}.
\end{lemma}

\begin{proof}
    If $n=0$, then $B_\w$ is the set of elements 
    $a$ satisfying $\nu(f_0-a)>0$. 
    The result follows in that case. If $n>1$, we can 
    assume, as an inductive hypotheses, that 
    $B_{\w'}=\eta_{n-1}^{-1}\left(B_0^{(0)}\right)$, 
    where $\w'=[f_0,\dots,f_{n-1}]$. Rewrite this as
    $B_0^{(0)}=\eta_{n-1}\left(B_{\w'}\right)$,
    so $\rho_{n-1}\left(B_{\w'}\right)$ is 
    the complement of the
    closed ball $B_0^{[0]}$. Now an element
    $z'\in B_{\w'}$ is in $B_\w$ if and only if
    $\nu\left(\rho_{n-1}(z')-f_n\right)>0$.
    Equivalently, we have
    $\rho_{n-1}(B_\w)=B_{f_n}^{(0)}$.
    In other words $\eta_n(B_\w)=\rho_{n-1}(B_\w)-f_n=B_0^{(0)}$.
    The first statement now follows since Moebius transformations
    map balls to either balls or complements 
    (in $\mathbb{P}^1(\hat{K})$) of balls,
    and $B_\w$ does not contain $\infty$. It is also clear that
    $a\in B_\w$, so all that remains is to compute the radius.
    Lemma \ref{lemma equal cf} proves that $B_a^{(r)}$ is contained
    in $B_{\w}$. For the converse, every element in $B_{\w}$
    has $a$ as an approximant, so the result follows from
    Equation (\ref{eq approx cf}).
\end{proof}

Recall that, according to Theorem \ref{main teo 1}, 
the rational $r$ above can be written as
$r=2\dg(q_n)$, where $\frac{p_n}{q_n}=\w_{\mathrm{ev}}=x_n(z)$
is the $n$-th approximant for every $z\in B_{\w}$. Since every
element of a ball can be regarded as the center, we can 
also write $B_{\w}=B_z^{[2\dg(q_n)]}$, for every
$z\in B_{\w}$.

\begin{proposition}\label{mt2b}
When $z\in\tilde{k}$, the associated continued fraction always ends. 
\end{proposition}

\begin{proof}
Write  $z=\frac pq$. Whenever the associated continued fraction 
fails to end, we have $\nu\left(\frac pq-\frac{p_n}
{q_n}\right)>2\dg(q_n)$
for arbitrarily large values of $\dg(q_n)$. However,
if $\dg(q_n)>\dg(q)$ we conclude
$$\dg(qp_n-pq_n)=
\dg(q)+\dg(q_n)-\nu\left(\frac pq-\frac{p_n}{q_n}\right)
<\dg(q)-\dg(q_n)<0,$$ a contradiction
unless $qp_n-pq_n=0$, and therefore $\frac pq=\frac{p_n}{q_n}$.
\end{proof}

\begin{proof}[Proof of Theorem \ref{main teo 4}]
Write  $a=\frac pq=\w_{\mathrm{ev}}$, where
$\w=[f_0,\dots,f_n]$, as we can always do by the preceding proposition. Then the hypotheses
 $$\nu\left(z-\frac{p}{q}\right)>-2\nu(q)=2\dg(q)=
 \sum_{i=1}^n\dg(f_i)=r,$$
 is equivalent to $z\in B_a^{(r)}=B_{\w}$.
Since the latter is, by definition, the set
of elements whose associated continued fraction start as
$[f_0,f_1,\dots,f_n,\dots]$, the element $a=\frac{p}{q}$ 
is an approximant of $z$, which concludes the proof.
\end{proof}

\begin{xrem}\label{main teo 5}
An element $f \in \hat{K}$ is called algebraic of degree $n$ over 
$\Tilde{k}$ when it is a zero of a polynomial in $\Tilde{k}[T]$ of
degree $n$.
When $n=2$, we say that $f$ is quadratic over $\Tilde{k}$.
One can give a straightforward generalization of the usual
characterization of quadratic element in terms of their continued fractions:
\begin{quote}
An element $z \in \hat{K}$ is quadratic over $\Tilde{k}$ if and only 
if its associated continued fraction is eventually periodic.
\end{quote}
In fact, assuming that $z \in \hat{K}$ is quadratic over $\Tilde{k}$,
then $z \in \K$, since $\Tilde{k}\subset \K$
and $\K$ is algebraically closed.
Then, it follows from Lemma \ref{lemma Puiseux field} that $z \in K_N$, for some $N \in \mathbb{Z}_{>0}$.
Hence, it follows from \cite{Schmidt} or \cite[\S 2]{Paulin} that $z$ can be written as a periodic continued fraction with coefficients in $A_N$. This is the unique continued fraction of $z$ according to Prop. \ref{prop cf are uniq}.
On the other hand, if $z \in \hat{K}$ has a periodic continued fraction, then $z \in K_N$, for some $N \in \mathbb{Z}_{>0}$, according to Corollary \ref{coro deg with common den}.
Then, \cite{Schmidt} implies that $z$ is quadratic over $k_N=F(t^{1/N}) \subset \Tilde{k}$.
\end{xrem}

\section{On the modular ray in a Berkovich space}\label{sec geom int of cf}

A geometrical interpretation for continued fractions
over the completion at infinity of a polynomial ring
was given in \cite{Paulin}. The purpose of this
section is to give an analogous construction for our context.
More specifically, here we give a geometrical interpretation
of Theo. \ref{main teo 1} in terms of the action of 
$\mathrm{SL}_2(\A)$ on the Berkovich projective line 
over $\hat{K}$.
This approach allows us to prove Theo. 
\ref{main teo 3} in next section.

Let $\mathbb{P}^{1,\mathrm{Berk}}$ be 
the Berkovich projective line over $\hat{K}$ as defined 
in \cite[Ch. II]{BakerBerkovich}.
Specifically, we focus on the subset of points of type II 
or III, which is classically denoted $\h$ in literature.
The space $\h$ can be constructed as the quotient 
$$(\hat{K}\times\mathbb{R}) / \sim, \text{ where } 
(a,r) \sim (a',r) \Leftrightarrow \nu (a-a') \geq r.$$
The class of $(a,r)$ is denoted $\eta_{a,r}$.
If $\eta=\e$, the valuation $\nu(\eta):=r$ is well
defined, while the set $B_\eta=
\{b\in \hat{K}|\eta=\eta_{b,r}\}$ is the closed ball
$B^{[r]}_a=\{b\in\hat{K}|\nu(b-a)\geq r\}$. 
Note that every closed
ball in $\hat{K}$ has this form.
Following \cite[Ex. 1.4.3]{Berkovich}, the point 
$\e$ is called a point
of type II if $r\in\nu(\hat{K})=\mathbb{Q}$, and of type III otherwise.

We endow $\hat{K}\times\mathbb{R}$ with the product topology,
and $\h$ with the quotient topology. Note that a sequence  
$\left \lbrace  \eta_{a_n,r_n}  \right \rbrace_{n=1}^{\infty} \subset \h$ 
converges to $ \e $ if and only if $\lim_{n \to \infty} r_n = r$ 
and $\liminf_{n \to \infty}\nu(a_n - a) \geq r$.
This topology is metrizable.
In fact, if we write $\eta=\e$ and $\eta'=\eta_{a',r'}$,
for an arbitrary pair $(\eta,\eta')\in \h \times \h$, and define
\begin{equation}\label{metriceq}
    d(\eta,\eta')=\begin{cases}
\lvert r-r' \rvert, & \text{ if } \nu(a-a') \geq  \mathrm{min}(r,r')\\
 r + r'- 2\nu(a-a') , & \text{ if } \nu(a-a') <  \mathrm{min}(r,r'),
\end{cases}
\end{equation}
then $d$ is a metric and it defines the preceding topology 
on $\h$.The map $d$ can be geometrically interpreted 
as follows: Given two points $\eta=\e$ and
$\eta'=\eta_{a',r'}$ in $\h$, we define 
$\eta \vee \eta'$ as $\eta_{a, r''}=\eta_{a',r''}$, 
where $r''=\min \lbrace r,r', \nu(a-a') \rbrace$. 
The (unique) geodesical segment joining $\eta$ with 
$\eta'$ is then 
$$[\eta, \eta']= \lbrace \eta_{a,s}, \eta_{a',s'}
\lvert r'' \leq s \leq  r,\, r''\leq s'\leq r'\rbrace.$$
Thus $d(\eta,\eta')$ is the length of an interval 
$[0, d(\eta,\eta')]\subseteq\mathbb{R}$ that is 
isometric to the segment $[\eta,\eta']\subset \h$. 

The metric space $\left(\h,d \right)$ is an 
$\mathbb{R}$-tree according to \cite[2.2]{BakerBerkovich}, 
i.e., for any pair of points $\eta,\eta' \in \h$ there is 
a unique segment from $\eta$ to $\eta'$, namely 
the segment $[\eta, \eta']$ defined above, and this segment 
is geodesic.
The nontrivial points at the completion
$\mathbb{H}^{\mathrm{Berk}}$ of $\left(\h,d \right)$, 
which are called type IV points,
can be characterized as limits of sequences
$\{\eta_i\}_{i=1}^\infty \subset \h$ 
whose corresponding balls
$B_{\eta_i}=B_{a_i}^{[r_i]}$ form a decreasing 
sequence $B_{\eta_1}\supseteq B_{\eta_2}\supseteq\cdots$ 
of closed balls with empty intersection 
(cf. \cite[\S 1.2]{BakerBerkovich} or 
\cite[Ex. 1.4.3]{Berkovich}).
Here two such sequences define the same type IV point exactly 
when they are cofinal.
Note that the completeness 
of $\hat{K}$ forces the corresponding sequence of 
rational numbers $\{r_i\}_i$ to converge to a finite limit.

A ray $\mathscr{R}$ in $\mathbb{H}^{\mathrm{Berk}}$ is the image $\mathscr{R}=\mathrm{Im}(p)$ of an isometry $p: [0,\infty) \to \mathbb{H}^{\mathrm{Berk}}$.
We say that two rays in $\mathbb{H}^{\mathrm{Berk}}$ are equivalent exactly when their intersection is a ray.
The equivalent class of a ray is called its visual limit, and
it is denoted $p(\infty)$ by an abuse of notation.
The Berkovich projective line $\mathbb{P}^{1,\mathrm{Berk}}$ 
is the compactification of $\mathbb{H}^{\mathrm{Berk}}$ defined by adding the visual limit of all its rays (cf. \cite[\S 3.5]{IntroBerk}).
The points of the visual limit are as follows:
One point $a^{\star}=\eta_{a,\infty}$ for each 
element $a\in\hat{K}$, and a common element 
$\infty^{\star}=\eta_{a,-\infty}$
for any $a$. These points are in correspondence with $\mathbb{P}^1(\hat{K})$,
and are called type I points.

The group $G:=\mathrm{GL}_2(\hat{K})$ acts via isometric maps on $\h$ according to \cite[2.13 \& 2.15]{BakerBerkovich}. 
We write $g * \e$ for the image of $\e \in \h$ via the action of $g \in G$.
Moreover, if we write 
\begin{equation}
\mathrm{i}=\sbmattrix{0}{1}{1}{0} \quad \text{ and } \quad \mathrm{m}_{d_1,d_2,f}=\sbmattrix{d_1}{f}{0}{d_2},  \quad d_1,d_2 \in \hat{K}^{*}, f \in \hat{K},
\end{equation}
then the preceding action can be described via:
\begin{equation}
\mathrm{i} * \e=\begin{cases}
\eta_{\frac{1}{a} , r-2\nu(a)} , & \text{ if } \nu(a) < r,\\
\eta_{ 0 , -r }, & \text{ if } \nu(a) \geq r,
\end{cases}
\end{equation}
and
\begin{equation}
\mathrm{m}_{d_1,d_2,f} * \e = \eta_{a',r'}, \text{ with } a'=\frac{d_1 a+f}{d_2} \text{ and } r'= r+\nu\left(\frac{d_1}{d_2}\right).
\end{equation}
In the sequel, we write $\mathrm{t}_f:=\mathrm{m}_{1,1,f}$.
The group $\mathrm{SL}_2(\A)$ acts on $\h$ as a subgroup 
of $G$.

\begin{lemma}\label{lemma closing umbrellas}
Let $a\in\Hat{K}$. Then, there exists $f\in \A$ 
such that $\mathrm{t}_{-f}*\e=
\eta_{ 0, r}$ for every $r\leq0$. Furthermore,
$f\neq0$ precisely when $\nu(a)\leq 0$, and in this case
$\nu(f) = \nu(a)$.
\end{lemma}

\begin{proof}
It follows from Lemma \ref{lemma first approx} that there exists 
$f\in \A$ satisfying the inequality $\nu(a-f)>0$.
Then $\mathrm{t}_{-f}*\e= \eta_{ a-f,r} =\eta_{ 0,r }$,
where the latter identity follows from $\nu(a-f)>0 \geq r$ 
via dominance principle.
The last statement is immediate from 
Lemma \ref{lemma first approx}.
\end{proof}

The geometrical interpretation of continued fractions
mentioned earlier makes use of  the Bruhat-Tits tree of
$\mathrm{SL}_2$ at a discretely valued field $\kappa$.
In the current literature, there exists more than one
interpretation of this tree (c.f. \cite[Chap. II, \S 1]
{SerreTrees} or \cite{BT1}).
For our purposes, we focus on the following realization, 
which comes from the topological structure of $\kappa$:
\begin{quote}
The vertex set $V_\kappa$ of $\mathfrak{t}_{\kappa}$
corresponds to the set of closed balls 
in ${\kappa}$, while the edge set $E_{\kappa}$ of
$\mathfrak{t}_{\kappa}$ corresponds to the pairs of 
closed balls where one is a maximal proper sub-ball of 
the other.
\end{quote}
The definition of the $\mathrm{GL}_2(\kappa)$-action 
on $\mathfrak{t}_{\kappa}$ via simplicial maps,
which can be found in \cite[\S 4]{ArenasArenasContreras},
is entirely analogous to the one recalled in this section 
for the Berkovich space. Moreover,
for each finite extension $K_n$ of $K$,
the topological realization (or polyhedron) 
$\mathbb{T}(K_n)$ of the Bruhat-Tits tree 
$\mathfrak{t}_{K_n}$ 
can be embedded into $\h$, in a way that the 
vertex corresponding to a ball
$B^{[r]}_{a}(K_n)=B_a^{[r]}\cap K_n$ is mapped onto the 
point $\eta_{a,r}$.
This embedding is $\mathrm{GL}_2(K_n)$-equivariant, as 
follows from \cite[\S 4]{ArenasArenasContreras} or 
\cite[Pag. 214]{CBHecke}.
In the sequel, we identify $V_{K_n}$ with 
$\left\lbrace \eta_{a,r}| 
a\in K_n , \,\, r \in 
\mathbb{Z}/n \right\rbrace$.
Analogous identification applies for $E_{K_n}$.

For any two elements
$\eta=\e$ and $\eta'=\eta_{a',r'}$ in $\h$
we write $\eta \succ \eta'$ (or equivalently $\eta' \prec \eta$) 
whenever $r$ exceeds neither $r'$ nor $\nu(a-a')$,
so that, in particular, $\eta=\eta_{a',r}$.
In this case, we say that $\eta$ lies above $\eta'$,
or that $\eta'$ lies below $\eta$.
We denote by $\V^r_{\downarrow}(\eta)$ the set of points below
$\eta$ and at distance $r$ from it.
The modular ray $\mathscr{R}_{\infty}$ of $\h$ is
the subspace 
$\lbrace \eta_{0,r}: r \leq 0 \rbrace\subset \h$.
This ray plays a fundamental role in the theory as next 
result shows:

\begin{proposition}\label{prop fund domain and closing umbrellas}
We have:
\begin{itemize}
\item[(a)] For each $r\leq 0$ and $0<s \leq -r$, the stabilizer
$\mathrm{Stab}_{\mathrm{SL}_2(\A)}(\eta_r)$ of $\eta_r= 
\eta_{0,r}$ acts transitively on 
$\V_{\downarrow}^s \left( \eta_r \right)$, and
\item[(b)] the ray $\mathscr{R}_{\infty}$ is a 
fundamental domain for the action of 
$\mathrm{SL}_2(\A)$ on $\h$.
\end{itemize}
\end{proposition}

\begin{proof}
Let $\eta':= \eta_{a,r'}$ be a point in $\V_{\downarrow}^s(\eta_r)$.
Since $\eta'$ lies below $\eta_r$, we have that $s=d(\eta_r,\eta')= r'-r$.
In particular, we have $r' \leq 0$.
Moreover, since $\eta_r\succ \eta'$, then $\e= \eta_{0,r}$, and  
therefore $\nu(a) \leq r\leq0$.
Then, Lemma \ref{lemma closing umbrellas} implies that 
$\mathrm{t}_{-f} * \eta'= \eta_{0,r'}=\eta_{0,s+r }$, 
for certain $f \in \A$ with $\nu(f)=\nu(a)\leq r$. In particular,
$\mathrm{t}_{-f}$ fixes $\eta_r$, and statement (a) follows.

Now, we prove statement (b). Let $\eta=\e \in \h$.
If $r\leq 0$ the result follows from the preceding paragraph, so we assume
$r>0$. Let $\mathrm{App}(a)=\{(f_i,a_i)\}_{i=0}^\infty$. Since
$\nu(a-f_0)=\dg(f_1)$, statement (b) clearly holds for $r\leq\dg(f_1)$, since $\mathrm{i}*(\mathrm{t}_{-f_0}*\eta)$ is in the ray.
Note that there exists
$n\in\mathbb{Z}_{\geq0}$ such that
$$\dg(f_{n+1})+2\sum_{i=1}^n\deg f_i\leq r\leq 
\dg(f_{n+2})+2\sum_{i=1}^{n+1}\deg f_i.$$
If $\rho_n$ corresponds to the matrix $\mathrm{r}_n$,
then $\mathrm{r}_n*\eta= \eta_{a_{n+1},r'}$
with $r'=r+\nu\big(\rho_n'(a)\big)=
r-2\sum_{i=1}^{n+1}\deg f_i$. In particular, we have
$-\dg(f_{n+1})\leq r'\leq \dg(f_{n+2})$, so the same argument applies.
We conclude that $\eta$ is in the $\mathrm{SL}_2(\A)$-orbit of some point 
in $\mathscr{R}_{\infty}$.

Now, we prove that no two points in $\mathscr{R}_{\infty}$ 
belong to the same $\mathrm{SL}_2(\A)$-orbit.
Indeed, note that, by \cite[\S 1.6, Ch. II]{SerreTrees},
the subgraph whose vertices
corresponds to the balls $B_0^{[-i/n]}$, for 
$i=0,1,2,\dots$ is a fundamental
region for the $\mathrm{SL}_2(A_n)$-action on the 
graph $\mathfrak{t}_{K_n}$.
Therefore, the corresponding ray is a fundamental region for
the action of $\mathrm{SL}_2(A_n)$ on the topological 
realization $\mathbb{T}(K_n)$.
Now assume there is a matrix $\mathrm{g}\in\mathrm{SL}_2(\A)$
satisfying $\mathrm{g}*\eta_r=\eta_{r'}$. Then 
$\mathrm{g}\in\mathrm{SL}_2(A_n)$ for some 
$n\in\mathbb{Z}_{>0}$.
Furthermore, $\mathscr{R}_{\infty}$ is contained 
in $\mathbb{T}(K_n)$,
when the later is identified with a subspace of $\h$.
In particular, we have $\mathrm{g}*\eta_r=\eta_{r'}$ for two points
$\eta_r,\eta_{r'}$ in a fundamental region for the action
of $\mathrm{g}\in\mathrm{SL}_2(A_n)$ on $\mathbb{T}(K_n)$.
We conclude that the points coincide, whence $r=r'$
and the result follows.
\end{proof}

\begin{xrem}
    In the second part of the preceding proof 
    we reduced the problem to a classical result,
    instead of writing an independent proof
    since the usual dimension argument used
    by Serre, in the given reference, 
    does not carry well to
    our setting, as the corresponding 
    dimensions are infinite. We have not been
    able to find a direct proof thus far.
\end{xrem}

\begin{lemma}\label{lemma action on the mod ray}
For every point $\eta\in\h$, and for 
    every sufficiently small $\epsilon>0$, every point 
    $\eta'\in\h$ satisfying $d(\eta,\eta')<\epsilon$ must satisfy 
    $d(y,y')<\epsilon$, where $y\in\mathscr{R}_{\infty}$ is 
    the point with the same image as $\eta$, and $y'$ is defined
    analogously.       
\end{lemma}

\begin{proof}
    Replacing $\eta$ by another element in 
    the same orbit we can assume
    that $\eta=y\in\mathscr{R}_{\infty}$. Assume therefore $\eta=\eta_{0,r}=\eta_r$ for $r\leq 0$. Let 
    $\eta'= \eta_{ a',r'}$ be a point with $d(\eta,\eta')<\epsilon$.
    Then Equation (\ref{metriceq}) shows that
    $|r-r'|\leq\epsilon$. Indeed, this is immediate if
    $\nu(a')\geq\mathrm{min}(r,r')$, so we assume
    $\nu(a')<\mathrm{min}(r,r')$, and we write
    $$d(\eta,\eta')=
    \mathrm{max}(r,r')+\mathrm{min}(r,r')-2\nu(a')>
    \mathrm{max}(r,r')-\mathrm{min}(r,r')=|r-r'|.$$
    When $r<0$, there is a neighborhood of $x$ where $\eta'= \eta_{ a,r'}$ implies $y'=\eta_{0,r'}=\eta_{r'}$,
    and therefore $d(y,y')=|r-r'|<\epsilon$,
    so we are left with the 
    case $r=0$.

    Now assume $\eta=y=\eta_{0,0}$. Then we write
    $y'=\mathrm{g}*\eta'= \eta_{ 0,r''}$, for some 
    $\mathrm{g}\in \mathrm{SL}_2(\A)$.
    Set $y_1=\mathrm{g}*\eta$.
    Note that $d(y',y_1)=d(\eta',\eta)<\epsilon$.
    If $r''<-\epsilon$, the preceding case, 
    with $y'$ instead of $\eta$, tells us that 
    $y_1=\eta_{a''',r'''}$ 
    with $|r'''-r''|<\epsilon$. In particular,
    $r'''<0$, which is absurd, since no element
    of the form $\eta_{ a''',r'''}$
    with $r'''<0$ can be in 
    the same orbit as $\eta_{0,0}$. We conclude that
    $d(y,y')=-r''<\epsilon$ and the result follows.  
\end{proof}

\begin{proposition}\label{prop iso with the ray}
    The canonical projection $\pi:\h\rightarrow\mathrm{SL}_2(\A)\backslash \h$ induces a homeomorphism 
    $$\tilde{\pi}:\mathscr{R}_{\infty}\rightarrow
    \mathrm{SL}_2(\A)\backslash \h.$$
\end{proposition}

\begin{proof}
    The fact that $\mathscr{R}_{\infty}$ is a fundamental region 
    tells us that the map $\tilde{\pi}$ is bijective, and it is certainly
    continuous. It suffices, therefore, to prove that 
    $\tilde{\pi}^{-1}$ is continuous. By the universal
    property of the quotient, it suffices to prove that 
    $\tilde{\pi}^{-1}\circ\pi:
    \h\rightarrow\mathscr{R}_{\infty}$
    is continuous. A function is continuous 
    if it is continuous at every point. Therefore, 
    the result follows from Lemma \ref{lemma action on the mod ray}.
\end{proof}

\section{Continued fractions associated to Type IV points}\label{sec cf for points of type IV}

A descending isometric geodesic (or DIG) is a map 
$\tilde{w}:(-\infty,s)\rightarrow\h$, for 
$s\in\mathbb{R}\cup\{\infty\}$,
that satisfies the relation $\nu\big(\tilde{w}(r)\big)=r$, and
is maximal, in the sense that cannot be extended to a larger 
open interval.
When $s=\infty$, there is a unique element $a\in\hat{K}$
for which the DIG can be written as 
$r\mapsto \e$ on the whole domain.
 In fact,
$a$ can be characterized as the only point $a$
belonging to every ball $B_{\tilde{w}(r)}$.
In this case the type I point $a^{\star}$ is called the
limit at $\infty$ of the DIG, and we say that
$\tilde{w}(r)$ converges to $a^{\star}$.
We call $\tilde{w}$ a full DIG in this case.
Conversely, for every element $a\in\hat{K}$
we can define a full DIG $\tilde{w}_a: \mathbb{R} \to \h$ by the formula
$\tilde{w}_a(t)=\eta_{a,t}$.
We might write $\tilde{w}_{a^{\star}}$ instead of
$\tilde{w}_a$, for convenience, for example to treat
type I and type IV points simultaneously.

On the other hand, since 
$r\mapsto \tilde{w}(r)$ is an isometry by definition, 
a non-full DIG must  converge to a point $\tilde{w}(s)$ 
in the completion $\mathbb{H}^{\mathrm{Berk}}$ of $\h$. 
Moreover, if $\tilde{w}(s)=\eta_{a,s} \in\h$,
then $\tilde{w}$ can be extended to all of $\mathbb{R}$ by setting
$\tilde{w}(r)=\e$ for $r>s$. Therefore, 
by the maximality condition, a non-full DIG
converges to a type IV point. Since for every type IV point
there is a unique geodesic connecting it to $\infty^{\star}$,
there is a unique such DIG corresponding to any type IV
point. Analogously to the case of type I points, for every 
type IV point $b$ we denote the corresponding
DIG by $\tilde{w}_b$.

Now, let $\pi:\h\rightarrow\mathrm{SL}_2(\A)\backslash \h$
be the canonical projection. For every DIG $\tilde{w}$ we would
like to study the corresponding trajectory, or promenade,
in $\mathrm{SL}_2(\A)\backslash \h$. This can be defined 
as the composition $\pi\circ\tilde{w}:(-\infty,s) \to \mathrm{SL}_2(\A)\backslash \h$.
It follows from the Proposition \ref{prop iso with the ray}
that we can regard such composition as a promenade on the 
ray $\mathbb{R}_{\geq0}\cong\mathscr{R}_{\infty}$, 
so we can study it as a 
real valued function. To make this precise, we define 
$\pi_0:
\h\rightarrow\mathbb{R}_{\geq0}$ as the composition
$\pi_0=\phi\circ\tilde{\pi}^{-1}\circ\pi$, where $\phi$
is the natural homeomorphism $\phi:\mathscr{R}_{\infty}
\rightarrow\mathbb{R}_{\geq0}$ defined by
$\phi\big(\eta_{0,r})=-r$.
For any isometrical geodesic $\tilde{v}:
(t_1,t_2)\rightarrow\h$ we define the associated 
real valued function (RVF) as
$v=\pi_0\circ\tilde{v}:(t_1,t_2) \to \mathbb{R}_{\geq 0}$.

\begin{lemma} \label{p61}
Let $v$ be the associated RVF of an isometrical geodesic
$\tilde{v}:(t_1,t_2)\rightarrow\h$.
Assume $v$ fails to take the value $0$
in the interval $(t_1,t_2)$. 
Then $v$ is either monotonic with slope
$\pm1$ on the whole interval, or it has a unique local maximum,
while being monotonic with slope $\pm1$ on either side.
\end{lemma}

\begin{proof}
    Replacing $\tilde{v}$ by a geodesic of the form 
    $\mathrm{g}*\tilde{v}$, for 
    $\mathrm{g}\in\mathrm{SL}_2(\A)$, we can assume that
    $\nu\big(v(t)\big)<0$ for at least one value of $t\in(t_1,t_2)$.
    Note that $\pi_0\big(\e\big)=-r$
    whenever $r\leq0$, according to Lemma 
    \ref{lemma closing umbrellas}. In particular,
    every element of
    the form $\eta_{a,0}$ is in the orbit of 
    $\eta_{0,0}$, so we conclude 
    that $\nu\big(v(t)\big)<0$ for every value
    $t \in (t_1,t_2)$.
    Now the result follows from the explicit description of 
    the geodesic between two points and the preceding formula
    for $\pi_0$.
\end{proof}

\begin{lemma}\label{p62}
Let $\tilde{v}:(t_1,t_2)\rightarrow\h$ be an 
isometrical geodesic
whose associated RVF $v$ is increasing
in the interval $(t_1,t_2)$. Then $\tilde{v}$ 
can be extended
to an interval of the form $(t_1,\infty)$.
\end{lemma}

\begin{proof}
     Replacing $\tilde{v}$ by another geodesic in the 
     same $\mathrm{SL}_2(\A)$-orbit as before, we can assume that the image
    of $\tilde{v}$ has a point on the ray $\mathscr{R}_{\infty}$.
    In other words, $\tilde{v}(t)=\eta_{ 0,-r}$, with $r>0$,
    for some $t\in(t_1,t_2)$. Since $v$ is increasing, we have
    $\tilde{v}(t')=\eta_{ 0,-r-(t'-t) }$ for every $t'>t$.
    It can be extended by the same formula.
\end{proof}

\begin{lemma}\label{p63}
Let $\tilde{v}$ and $v$ as before, and assume $v$
takes the value $0$ at some 
point $t\in(t_1,t_2)$. Then there exists $\epsilon>0$ such
that $v$ is monotonic with slope $\pm1$ in either interval, 
$(t-\epsilon,t)$ and $(t,t+\epsilon)$.
\end{lemma}

\begin{proof}
    As before, we can assume that $\tilde{v}(t)=\eta_{0,0}$, for some $t\in (t_1,t_2)$.
    Let $t_3\in (t_1,t)$ and $t_4\in (t,t_2)$, and set
    $\tilde{v}(t_i)=\eta_{ a_i,r_i }$, for $i=3,4$.
    If $\nu(a_i)<0$ or $r_i<0$, the geodesic connecting
    $t_i$ and $t$ contain points in the modular ray.
    This cannot happen simultaneously for $i=3$ and
    $i=4$, as $\tilde{v}$ is a geodesic passing through
    $\eta_{0,0}$. It follows that at least one point,
    $\tilde{v}(t_3)$, or $\tilde{v}(t_4)$,lies below 
    $\eta_{0,0}$. If $r_3<0$, we apply a Moebius transformation 
    of the form $z\mapsto\frac1{z-b}$, where 
    $b\in\mathcal{O}_{\hat{K}}$
    is not congruent to $a_4$ modulo the maximal ideal 
    $\mathcal{M}$. The case where $r_4<0$ is handled similarly,
    so we might assume both $r_3,r_4>0$
    and $a_3$ in not congruent to $a_4$ modulo $\mathcal{M}$.
    Choose elements $c_3,c_4\in F$ satisfying $c_i\equiv a_i$
    modulo $\mathcal{M}$, for $i=3,4$. This implies that
    $\nu_i:=\nu(a_i-c_i)>0$. Choose $\epsilon<
    \min(\nu_3,\nu_4,t-t_3,t_4-t)$. Then,
    the Moebius transformation $z\mapsto\frac1{z-c_3}$
    maps the geodesic between 
    $\tilde{v}(t-\epsilon)=\eta_{ c,\epsilon}$ and
    $\tilde{v}(t)=\eta_{c,0}$ to the geodesic between
    $\eta_{ 0,-\epsilon}$ and $\eta_{ 0,0}$.
    The interval $(t,t+\epsilon)$ is handled similarly.
\end{proof}

In what follows we write $w_u=\pi_0\circ\tilde{w}_u$
for the promenade corresponding to every point $u$ of 
type I or IV. This is a real valued function defined on an
interval $(-\infty,s)$.  When $u=a^{\star}$ is a type I point, 
then $s=\infty$. For $t<0$ we have 
$w_u(t)=\pi_0\big(\eta_{a,t}\big)=-t$. 
At $t=0$ the promenade bounces back at $0$ 
(Prop. \ref{p63}), and
then, according to Prop \ref{p61},
either proceeds to return back to $\infty$ or reaches
a maximum, returning to $0$ once again. This process can either
repeat indefinitely, or ends by the promenade heading towards 
$\infty$ after a finite number of repetitions.
The situation is similar for a type IV point, except that
in this case the promenade must stop at a finite value $s$
of the parameter.

\begin{proposition}\label{p 74}
For a type I point $u=a^{\star}$, the value at the 
successive maxima of the
function $w_u$ coincide with the degrees $\dg(f_i)$ 
for $i\in\mathbb{Z}_{>0}$. In particular, the promenade proceeds to
$\infty$ after a finite number of iterations precisely when 
the continued fraction of $a$ stops, i.e., when $a\in\tilde{k}$. 
\end{proposition}

\begin{proof}
It is a straightforward computation that $\sigma_0(z)=\frac1{z-f_0}$
takes the geodesic from $\eta_{a,-\dg(f_0)}$ to 
$\eta_{ a,\dg(f_1)}$ to  the geodesic from 
$\eta_{ 0,\dg(f_0)}$ to $\eta_{0,-\dg(f_1)}$,
and also the geodesic from $\eta_{a,\dg(f_1)}$ to 
$a^{\star}$ to  the geodesic from 
$\eta_{ 0,-\dg(f_1)}$ to $\sigma_0(a)^{\star}$. 
When $f_1$ is not defined, then $\sigma_0$
takes the geodesic from $\eta_{a,-\dg(f_0)}$ to 
$a^{\star}$ to  the geodesic from 
$\eta_{ 0,\dg(f_0)}$ to $\infty^{\star}$.
Iterations are similar.
\end{proof}

\begin{proposition}\label{p66}
 Consider a sequence $\{f_n\}_{n=0}^{\infty}$ in $\A$ for which
    the series $\sum_{n=0}^\infty\dg(f_n)$ converges to a finite
    value. Then, if 
    $\w_n=[f_0,\dots,f_n]$, the sequence of balls 
    $\{B_{\w_n}\}_{n=1}^{\infty}$ has empty intersection.
\end{proposition}

\begin{proof}
 Assume $a\in B_{\w_n}$ for every $n$. Then the continued fraction 
 of $a$ must be $[f_0,f_1,\dots]$, but the hypothesis on the degrees
 contradicts Proposition \ref{Prop conv fc}.
\end{proof}

\begin{corollary}
Consider a sequence $\{f_n\}_{n=0}^{\infty}$ in $\A$ for
which the series $\sum_{n=0}^\infty\dg(f_n)$ converges to
a finite limit. Set $\w_n=[f_0,\dots,f_n]$. 
If $\eta_n \in \h$ is the point corresponding to the ball 
$B_{\w_n}$, then the sequence $\{\eta_n\}_{n=1}^{\infty}$ 
converges to a type IV point in the completion of $\h$.\qed
\end{corollary}

\begin{definition}
    A type IV point $\eta$ has type IVa if there is a sequence 
    $\{f_n\}_{n=1}^{\infty}$ in $\A$ satisfying
    the following conditions:
    \begin{enumerate}
    \item $\sum_{n=1}^{\infty}\deg{f_i}<\infty$.
    \item If $\w_n=[f_0,\dots,f_n]$, then  
    $B_{\w_n}\stackrel{n\rightarrow\infty}
    {\longrightarrow}\eta$.
    \end{enumerate}
    Otherwise we say that $\eta$ has type IVb.
\end{definition}

Our objective in the rest of this section is to give
a characterization for the points of type IVb.

\begin{example}\label{e69}
    Consider the elements $a_n=\sum_{i=1}^{n}t^{1/i}$,
    and the balls $B_n=B_{a_n}^{[\nu(t)/(n+1)]}$. Then
    $\{B_n\}_{n=1}^{\infty}$ is a sequence of nested balls
    whose intersection is $\varnothing$. The same holds 
    if we write $b_n=\sum_{i=1}^{N}t^{r(i)}$,
    and set $B_n=B_{b_n}^{[r(n+1)\nu(t)]}$ for
    any decreasing sequence of positive real numbers 
    $\{r(n)\}_{n=1}^{\infty}$. 
    It is immediate that the type
    IV point $b$
    corresponding to any such intersection has type IVb, 
    since  the promenade $w_b$ is defined on an interval of
    the form $(-\infty,s)$ for $s\leq0$, so $w_b$
    does not bounce even once.
\end{example}

\begin{definition}
Set $\A_{\dg>0}=\{f\in\A|\dg(f)>0\}$.
Let $\w=[f_0,\dots,f_n]$ be a finite expression with 
$f_i\in \A_{\dg>0}$.
    If $D$ is the ball corresponding to a point
    $\e\in \h$ with
$r\leq0$, then we write $B_{\w-D}$ for the set of
all elements in $\hat{K}$ whose continued fraction has the form
$[f_0,\dots,f_n,h_{n+1},g_{n+2},\dots]$ with $h_{n+1}\in D \cap \A_{\dg>0}$, where  $g_{n+2},\dots$ are arbitrary elements in  $\A_{\dg>0}$.
\end{definition}

\begin{proposition}\label{p610}
    The set $B_{\w-D}$ described above is a ball 
    in $\hat{K}$. Furthermore,
    any ball $B$ has the form $B=B_{\w-D}$ for some
    possibly empty expression $\w$, and for the 
    ball $D$ corresponding
    to some point $\eta=\eta_D\in\h$ satisfying $\nu(\eta)\leq0$.  
\end{proposition}

\begin{proof}
Let $\rho_n$ be as in Definition \ref{def rho}.
It is easy to see that $\rho_n(B_{\w-D})=D$. 
Note that the pole of $\rho_n$ is precisely
the element in $\tilde{k}$ whose continued
fraction equals $[f_0,\dots,f_n]$.
Then the first statement 
follows from the fact that Moebius 
transformations map any ball that
fails to contains its pole onto another ball.
Now take an arbitrary ball $B=B_a^{[r]}$, write $a=[f_0,f_1,\dots]_{\mathrm{ev}}$,
and find $n$ so that $2\sum_{i=1}^n\dg(f_i)<r\leq
2\sum_{i=1}^{n+1}\dg(f_i)$. Then it is easy to see that
$\rho_n(B)$ is a ball in $\hat{K}$
corresponding to a point $\eta \in \h$ with 
$\nu(\eta)\leq0$, which concludes the proof.
\end{proof}

\begin{definition}
  We write $\eta_{\w-D}\in\h$ for the point corresponding 
  to the ball $B_{\w-D}$. Similarly, we write $\eta_\w$
  for the point corresponding to $B_\w$.
\end{definition}

Note that, as follows from the relation 
$\rho_n(B_{\w-D})=D$, we have
$B_{\w-D_1} \supseteq B_{\w-D_2}$ whenever 
$D_1 \supseteq D_2$.
In particular, the limit of $\lbrace \eta_{\w-D_n} \rbrace_{n=1}^{\infty} \subset \h$, for $D_1\supseteq D_2 \supseteq \cdots$, is a point of type IV
whenever $\bigcap_{i=1}^\infty D_i=\varnothing$.

\begin{proposition}\label{p613}
    If $\eta$ has type IVb, then there is a longest 
    expression $\w$
    for which $\eta\subseteq B_{\w}$. Furthermore, 
    there is a nested sequence of balls 
    $D_1\supseteq D_2\supseteq\dots$, with
    $\nu(\eta_{D_i})<0$, for which 
    $\eta_{\w-D_n}\stackrel{n\rightarrow\infty}{\mapsto}\eta$.
\end{proposition}

\begin{proof}
Assume $\eta$ corresponds to a decreasing sequence 
$\{B_n\}_{n=1}^{\infty}$ of balls with empty intersection.
Write $B_n=B_{\w_n-D_n}$, as can be done for any ball by
Prop. \ref{p610}. If $\w_n\neq \w_{n+1}$, then necessarily
$\w_n$ is an initial segment of $\w_{n+1}$, whence
$B_n\supseteq B_{\w_{n+1}}\supseteq B_{n+1}$. 
It follows that,
if the expression $\w_n$ gets arbitrarily long as $n\rightarrow\infty$,
then $\eta$ can be seen as the limit of a sequence 
$\{\eta_{\w_n}\}_{n=1}^{\infty}$, and therefore 
it is a type IVa point. We conclude that the sequence
of expressions stabilize, whence the result follows.
\end{proof}

\begin{proof}[Proof of Theorem \ref{main teo 3}]
The first statement in the Theorem is in fact the first statement 
in Lemma \ref{l42}. The second statement is Proposition
\ref{p66}, and the last one follows from Example \ref{e69}
and Proposition \ref{p613}.
\end{proof}

\begin{xrem}
The group $\mathrm{SL}_2(\A)$ acts transitively on the set $\mathbb{P}^1(\Tilde{k})$, which is a proper subset of the set of type I points $\mathbb{P}^1(\hat{K})\subset
\mathbb{P}^{1,\mathrm{Berk}}$.
However, this action is non-transitive on the full set $\mathbb{P}^1(\hat{K})$.
In fact, our results on continued fractions can
be used to produce sequences of points in $\h$,
converging to a type I point,
whose image in the ray $\mathscr{R}_{\infty}$ is 
dense, for instance, considering a continued 
fraction whose coefficients have a dense set 
of degrees. This explain why we use the space 
$\h$ for study the action of $\mathrm{SL}_2(\A)$ 
instead the full space 
$\mathbb{P}^{1,\mathrm{Berk}}$.
\end{xrem}

\subsection*{Acknowledgements}
The authors would like to thank Lorenzo Fantini for several useful remarks.
The second author also thanks Anid-Conicyt for its partial support through the Postdoctoral fellowship No $74220027$.


\end{document}